\documentclass{amsart}
\usepackage{amsfonts}
\usepackage{graphicx}
\usepackage{tabularx}
\usepackage{array}
\usepackage[usenames,dvipsnames]{color}
\usepackage{comment}
\usepackage{amsmath}
\usepackage{amsthm}
\usepackage{amssymb}
\usepackage{fullpage}
\usepackage{xcolor}
\usepackage{tikz}

\tikzstyle{legend_general}=[rectangle, rounded corners, thin,
                          top color= white,bottom color=lavander!25,
                          minimum width=2.5cm, minimum height=0.8cm,
                          violet]

\DeclareMathOperator{\ch}{\chi_{\ell}}

\newtheorem{theorem}{Theorem}[section]

\newtheorem{lemma}[theorem]{Lemma}

\theoremstyle{definition}

\author{Peter Bradshaw}
\address{Simon Fraser University, Burnaby, BC, Canada}
\email{pabradsh@sfu.ca}

\title{Separating the online and offline DP-chromatic numbers}

\begin{document}
\maketitle
\begin{abstract}
The DP-coloring problem is a generalization of the list-coloring problem in which the goal is to find an independent transversal in a certain topological cover of a graph $G$. In the online DP-coloring problem, the cover of $G$ is revealed one component at a time, and the independent transversal of the cover must be constructed in parts based on incomplete information. Kim, Kostochka, Li, and Zhu asked whether the chromatic numbers corresponding to these two graph coloring problems can have an arbitrarily large difference in a single graph. We answer this question in the affirmative by constructing graphs for which the gap between the online DP-chromatic number and the offline DP-chromatic number is arbitrarily large.
\end{abstract}

\section{Introduction}
We will consider several graph coloring problems.
In the \emph{list coloring problem}, we have a graph $G$ and a list $L(v)$ of colors at each vertex $v \in V(G)$. In this setting, we say that an \emph{$L$-coloring} of $G$ is a proper coloring $\varphi:V(G) \rightarrow \bigcup_{v \in V(G)} L(v)$ of $G$ in which $\varphi(v) \in L(v)$ for every vertex $v \in V(G)$. If $G$ always has an $L$-coloring whenever $|L(v)| = f(v)$ for each vertex $v \in V(G)$, then we say that $G$ is \emph{$f$-choosable}. If $f$ is a constant function $f(v) = k$, then we say that the \emph{list-chromatic number} (or \emph{choosability}) of $G$ is at most $k$, and we write $\ch(G) \leq k$.

The \emph{DP-coloring problem} is a generalization of the list coloring problem introduced by Dvo\v{r}\'ak and Postle \cite{DP}, defined as follows.
Given a graph $G$ and a function $f:V(G) \rightarrow \mathbb N$, an \emph{$f$-fold cover} of $G$ is a graph $H$ obtained by the following process:
\begin{itemize}
\item For each vertex $v \in V(G)$, add a clique $K_{f(v)}$ to $H$, and write $L(v)$ for the vertex set of this clique.
\item For each edge $uv \in E(G)$, add a matching between $L(u)$ and $L(v)$.
\end{itemize}
Then, we say that an independent set in $H$ of size $|V(G)|$ is a \emph{DP-coloring} of $G$ with respect to $H$. If $G$ always has a DP-coloring for every $f$-fold cover $H$ of $G$, then we say that $G$ is DP-$f$-colorable. 
If $f$ is a constant function $f(v) = k$, then we say that $H$ is a \emph{$k$-fold} cover of $G$, and if $G$ always has a DP-coloring for every $k$-fold cover $H$ of $G$, then we say that the DP-chromatic number of $G$ is at most $k$, and we write $\chi_{DP}(G) \leq k$. 
Given a cover $H$ of $G$, we often refer to the vertices of $H$ as \emph{colors}, and if $c \in L(v)$ for a vertex $v \in V(G)$, then we say that the color $c$ is \emph{above} $v$. Note that when $f(v) = k$ is a constant function, if the cliques in $H$ corresponding to vertices in $G$ are replaced with independent sets, and if each matching between sets $L(u)$ and $L(v)$ is a perfect matching, then $H$ is a $k$-sheeted covering space of $G$, and a DP-coloring of $G$ is equivalent to an independent transversal of the fibers in $H$ above the vertices of $G$ (see \cite{Hatcher} for an introduction to graphs as topological spaces).

Every list-coloring problem can be transformed into a DP-coloring problem as follows. Given a graph $G$ with a color list $L'(v)$ at every vertex $v \in V(G)$, we construct a cover $H$ of $G$ by adding a clique with vertex set $L(v)$ for every vertex $v \in V(G)$, with elements of $L(v)$ corresponding to colors in $L'(v)$. Then, we consider each edge $uv \in E(G)$, and we add an edge in $H$ between each pair $(c,c') \in L(u) \times L(v)$ for which $c$ and $c'$ both correspond to a common color from $L'(u) \cap L'(v)$. When $H$ is constructed this way, a DP-coloring of $G$ with respect to $H$ is equivalent to an $L'$-coloring of $G$. Therefore, it holds that $\ch(G) \leq \chi_{DP}(G)$.

We will also consider two online graph coloring problems. The \emph{online DP-coloring} problem takes place in the form of a DP-coloring game between two players, called Lister and Painter. The game is played on a graph $G$ with a function $f:V(G) \rightarrow \mathbb N$. At the beginning of the game, each vertex $v \in V(G)$ has $f(v)$ tokens. On each turn $i$, Lister removes some number $m_i(v)$ (possibly zero) of tokens from each vertex $v \in V(G)$ and then reveals a clique $K_{m_i(v)}$ above each vertex $v$. Furthermore, for each edge $uv \in E(G)$, Lister reveals a matching between the revealed cliques above $u$ and $v$. The cliques and matchings revealed on this turn $i$ form a cover $H_i$ of $G$. After $H_i$ is revealed, Painter chooses an independent set from the vertices of $H_i$. The game ends when $G$ has no more tokens for Lister to remove. Painter wins the game if she manages to choose at least one color above each vertex of $G$ before the game is over; otherwise, Lister wins. If Painter always has a winning strategy in the DP-coloring game on a graph $G$ when each vertex $v \in V(G)$ begins with $k$ tokens, then we say that the \emph{online DP-chromatic number} (or \emph{DP-paintability}) of $G$ is at most $k$, and we write $\chi_{DPP}(G) \leq k$. Observe that if each vertex of $G$ begins with $k$ tokens, then Lister has the option of revealing a $k$-fold cover $H$ of $G$ on the first turn and asking Painter to find a DP-coloring of $G$ with respect to $H$, and therefore $\chi_{DP}(G) \leq \chi_{DPP}(G)$.

If, in the DP-coloring game, Lister is only allowed to remove at most one token from each vertex of $G$ during each turn and must always reveals edges wherever possible, then we call this variant of the game the \emph{list-coloring game}. For the list-coloring game, we may equivalently imagine that on each turn $i$, Lister reveals a single color $c_i$ above each vertex of some induced subgraph $G'_i$ of $G$, and Painter must choose some independent set $I_i$ of $G'_i$ and color each vertex in $I_i$ with $c_i$. 
In this equivalent setting, each vertex $v \in V(G)$ still begins with $f(v)$ tokens, and a single token is removed from $v$ whenever Lister reveals a color above $v$.
In this setting, Painter wins the game if and only if she can color every vertex of $G$ before the game ends.
 If Painter always has a strategy to win the list-coloring game on a graph $G$ when each vertex $v \in V(G)$ begins with $f(v)$ tokens, then we say that $G$ is \emph{$f$-paintable}. If $f$ is a constant function $f(v) = k$, then we say that $G$ is \emph{$k$-paintable}, and we write $\chi_P(G) \leq k$. The online list-coloring game was originally invented using this framework of revealing colors above vertices independently by Schauz \cite{Schauz} and Zhu \cite{ZhuP}.

At the end of the list-coloring game on $G$ with a constant function $f(v) = k$, the colors revealed at each vertex $v$ form a set $L(v)$ of $k$ colors, and if Painter wins the game, then Painter completes a proper $L$-coloring of $G$. Since Lister is free to form any list assignment $L$ on the vertices of $G$, it follows that if $G$ is $k$-paintable, then $G$ is also $k$-choosable, and hence $\chi_{\ell}(G) \leq \chi_P(G)$. Also, since the online list-coloring game is at least as difficult for Lister as the DP-coloring game, it also follows that $\chi_P(G) \leq \chi_{DPP}(G)$.

After putting all of the inequalities between these parameters together, we obtain two inequality chains:
\begin{eqnarray*}
 \chi_{\ell}(G)  \leq  & \chi_{DP}(G)  & \leq  \chi_{DPP}(G); \\
 \chi_{\ell}(G)  \leq &  \chi_{P}(G)  & \leq  \chi_{DPP}(G) .
\end{eqnarray*}
Given these inequality chains, it is natural to ask whether the differences between adjacent parameters can be arbitrarily large. 
For three out of these four differences, we find an affirmative answer by letting $G = K_{n,n}$. 
Indeed, Bernshteyn \cite{Bernshteyn} showed that a graph $G$ of average degree $d$ satisfies $\chi_{DP}(G) = \Omega \left ( \frac{d}{\log d} \right )$, implying that $\chi_{DP}(K_{n,n})  =  \Omega \left (\frac{n}{\log n} \right )$. 
Since it is known that $\chi_{\ell}(K_{n,n}) \leq \chi_P(K_{n,n}) = \log_2 n + O(1)$ \cite{Duraj}, this shows that
\[\chi_{DP}(K_{n,n}) - \chi_{\ell}(K_{n,n}) = \Omega \left ( \frac{n}{\log n} \right )\textrm{ \indent and \indent} \chi_{DPP}(K_{n,n}) - \chi_{P}(K_{n,n}) = \Omega \left ( \frac{n}{\log n} \right ).\]
 Duraj, Gutowski, and Kozik \cite{Duraj} also showed that  \[\chi_{P}(K_{n,n}) - \chi_{\ell}(K_{n,n}) = \Omega(\log \log n).\] 
Therefore, by letting $G = K_{n,n}$, we achieve an arbitrarily large difference for each adjacent pair of parameters except for $\chi_{DPP}(G) - \chi_{DP}(G)$. For this final difference, Kim, Kostochka, Li, and Zhu \cite{KKLZ} showed there exist graphs $G$ for which $\chi_{DPP}(G) - \chi_{DP}(G) \geq \chi_{P}(G) - \chi_{DP}(G) \geq 1$, implying that this final difference can be positive. However, it has not been shown that this difference can be arbitrarily large. 

In this note, we will show that the difference $\chi_{DPP}(G) - \chi_{DP}(G)$ can also be arbitrarily large, answering a question of Kim, Kostochka, Li, and Zhu \cite{KKLZ}. Rather than choosing $G = K_{n,n}$, we will construct a graph $G_t$ for each $t \geq 1$ that satisfies $\chi_{DPP}(G_t) - \chi_{DP}(G_t) \geq t$. We construct our graphs $G_t$ by generalizing an idea from the original paper of Kim, Kostochka, Li, and Zhu \cite{KKLZ}. Our graphs $G_t$ will also satisfy the additional property that $\chi_{P}(G_t) - \chi_{DP}(G_t) \geq t$.
By combining this equality with the already-known estimate $\chi_{DP}(K_{n,n}) - \chi_P(K_{n,n}) = \Omega \left ( \frac{n}{\log n} \right )$, we hence see that the difference $\chi_{DP}(G) - \chi_{P}(G)$ can achieve both positive and negative values of arbitrarily large magnitude.

\section{The construction}
For each integer $t \geq 1$, we will construct a graph $G_t$ that satisfies $\chi_P(G_t) - \chi_{DP}(G_t) \geq t$. Since $\chi_{DPP}(G) \geq \chi_P(G)$ for all graphs $G$, our graphs $G_t$ will also satisfy
 $\chi_{DPP}(G_t) - \chi_{DP}(G_t) \geq t$.
Our construction is based on a generalization of an idea of Kim, Kostochka, Li, and Zhu~\cite{KKLZ}.

As we are concerned with showing that the paintability of each graph $G_t$ is large enough, we begin with an observation about the online list-coloring game.
If Lister and Painter play the online list-coloring game on a graph $G$ with some initial token assignment, then Lister wins if and only if he can reach a position in which
each uncolored vertex $v \in V(G)$ has some $g(v)$ remaining tokens, and the uncolored subgraph of $G$ is not $g$-choosable.
In the original paper of Kim, Kostochka, Li, and Zhu~\cite{KKLZ}, the authors take advantage of this idea in order to construct a graph $G$ satisfying $\chi_P(G) \geq \chi_{DP}(G) + 1$.
In order to show that their graph $G$ has a large enough paintability, these authors show that in the online list-coloring game on their graph $G$, Lister always has a strategy to create an uncolored $K_{1,k}$ subgraph of $G$ in which each leaf $\ell$ has $g(\ell) = 1$ token and the center vertex $v$ has $g(v) = k$ tokens. Since $K_{1,k}$ is not $g$-choosable, it follows that Lister has a strategy to win the online list-coloring game on their graph $G$.

In our construction, we will use a similar idea. 
We first fix the value $k = 2^{8t^3}$. (With more careful calculation, our proof would work with a smaller value of $k$, but we use this larger value for clearer presentation.)
In each of our graphs $G_t$, we will show that Lister can always create an uncolored $K_{t,k^t}$ subgraph in which each $t$-degree vertex $u$ has $g(u) = t$ tokens and each $k^t$-degree vertex $v$ has $g(v) = k$ tokens.
The following lemma shows that if Lister manages to create such a subgraph of $G_t$, then Lister wins the online list-coloring game.
\begin{lemma}
\label{lem:ktkt}
Given the function $g: V(K_{t,k^t}) \rightarrow \mathbb N$ defined above, $K_{t,k^t}$ is not $g$-choosable. 
\end{lemma}
\begin{proof}
We let the $t$ vertices $v_1, \dots, v_t$ of degree $k^t$ have pairwise disjoint color lists $L(v_1), \dots, L(v_t)$ of size $k$. Then, for each of the $k^t$ elements $L \in L(v_1) \times \dots \times L(v_t)$, we let $L$ be the list of a vertex $u$ of degree $t$. Then, for any coloring of $v_1, \dots, v_t$ using colors from their lists, some vertex $u$ of degree $t$ has no available color in its list, and hence $K_{t,k^t}$ is not $g$-choosable.
\end{proof}

The most important piece of our construction of $G_t$ will be the following gadget $H_t$. We construct our gadget $H_t$ along with a function $h:V(H_t) \rightarrow \mathbb N$ as follows. We let $H_t$ contain $(t+1)k^t$ copies $K^1, \dots, K^{(t+1)k^t}$ of the clique $K_{t+1}$, and we write $u^{\ell}_1, \dots, u^{\ell}_{t+1}$ for the vertices of each clique $K^{\ell}$. We write $U$ for the set of all of these vertices of the form $u_j^{\ell}$; in other words, we let $U$ consist of all vertices that we have introduced so far. Then, for each value $1 \leq j \leq t+1$, we add $t$ independent vertices $x_j^1,\dots,x_j^t$, and we make each of these vertices adjacent to $u^1_j, u^2_j,\dots,u^{(t+1)k^t}_j$. We write $X$ for the set consisting of all of these vertices of the form $x_j^i$. For each vertex $u_j^{\ell} \in U$, we let $h(u_j^{\ell}) = t+1$, and for each vertex $x_j^i \in X$, we let $h(x_j^i) = k - t + 1$. We now prove two lemmas that show that under appropriate circumstances, winning the online list-coloring game on $H_t$ as Painter is much harder than finding a DP-coloring on $H_t$.

\begin{lemma}
\label{lem:no}
$H_t$ is not $(h+t-1)$-paintable. 
\end{lemma}
\begin{proof}
We give each vertex $v \in V(H_t)$ exactly $h(v) + t - 1$ tokens, and we show that Painter cannot win the list-coloring game on $H_t$.

On each of the first $t$ turns, Lister reveals a color at each vertex of each clique $K^{\ell}$ and reveals an edge wherever possible. After these first $t$ turns, each clique $K^{\ell}$ must have an uncolored vertex $u_j^{\ell}$ with exactly $h(u_j^{\ell}) - 1 = t$ tokens. Furthermore, since we have $(t+1)k^t$ cliques $K^{\ell}$, each with at least one uncolored vertex, there must exist some value $1 \leq j^* \leq t+1$ for which at least $k^t$ vertices of the form $u_{j^*}^{\ell}$ are uncolored. However, the $k^t$ vertices of the form  $u_{j^*}^{\ell}$ along with the $t$ vertices $x^1_{j^*}, \dots, x^t_{j^*}$ form an uncolored $K_{t,k^t}$ subgraph in which each $t$-degree vertex $v$ has only $g(v) = t$ remaining tokens, and each $k^t$-degree vertex $v$ has only $g(v) = k$ remaining tokens. Since Lemma \ref{lem:ktkt} shows that this $K_{t,k^t}$ subgraph is not $g$-choosable, Lister has a strategy to win the game.
\end{proof}

\begin{lemma}
\label{lem:yes}
$H_t$ is DP-$h$-colorable.
\end{lemma}
\begin{proof}
Consider an $h$-fold cover $H'$ of $H_t$. Recall that given a vertex $v \in V(H_t)$, we say that $v$ corresponds to a clique $K_{h(v)}$ in $H'$ with a vertex set $L(v)$, and we say that $L(v)$ contains $h(v)$ \emph{colors}. Using this terminology, we say that each color in $H'$ appears above only one vertex of $H_t$, as the sets $L(v)$ forming the cliques of $H'$ are pairwise disjoint.

We will first color the vertices $x^i_j \in X$. For each vertex $x^i_j \in X$ and color $c \in L(x^i_j)$, we assign a set $S_c \subseteq [(t+1)k^t]$ that consists of those indices $\ell$ for which $L(u_j^{\ell})$ contains a color adjacent to $c$. We would like to color the $t(t+1)$ vertices of $X$ using $t(t+1)$ colors $c_1, \dots, c_{t(t+1)}$ that correspond to a family $\mathcal S =\{ S_{c_1}, \dots, S_{c_{t(t+1)}}\}$ such that for each value $1 \leq q \leq t+1$, the following property holds:
\begin{equation}\tag{$\star$}
\label{star}
\textrm{The intersection of any $q$ sets of $\mathcal S$ contains at most $k^{t  - \frac{qt}{t+1}} - 1$ elements.}
\end{equation}
In particular, the intersection of any $t+1$ sets of $\mathcal S$ is empty.

We show that we may greedily color each vertex $x_j^i \in X$ while satisfying (\ref{star}). Suppose we wish to color some vertex $x \in X$ and that we have already colored some subset $Y \subseteq X$ while satisfying (\ref{star}). For each subset $A \subseteq Y$ of size $q \in [0,t]$ whose vertices are colored with colors $a_1, \dots, a_q$, we must choose some color $c \in L(x)$ for which 
\begin{equation}\tag{$\bullet$}
\label{eqn:int}
|S_c \cap S_{a_1} \cap \dots \cap S_{a_q}| \leq k^{t - \frac{(q+1)t}{t+1}} - 1.
\end{equation}
Note that since $h(u^{\ell}_j) = t+1$ for each vertex $ u^{\ell}_j \in U$, each value $\ell \in [(t+1)k^t]$ appears in at most $t+1$ sets $S_c$ for colors $c \in L(x)$. Furthermore, since $|S_{a_1} \cap \dots \cap S_{a_q}| < k^{t - \frac{qt}{t+1}}$ whenever $q \geq 1$, the number of colors $c \in L(x)$ that do not satisfy (\ref{eqn:int}) for a given $A \subseteq X$ is at most 
\[\frac{(t+1)k^{t - \frac{qt}{t+1}} }{k^{t - \frac{(q+1)t}{t+1}}} = (t+1)k^{\frac{t}{t+1}}.\]
Furthermore, since fewer than $2^{t(t+1)}$ subsets $A \subseteq Y$ can be chosen, the number of colors $c \in L(x)$ that would cause (\ref{star}) to be violated is less than
\begin{eqnarray*}
2^{t(t+1)}(t+1) k^{\frac{t}{t+1}} &= & 2^{t(t+1) + \frac{8t^4}{t+1}}(t+1) \\
 &=&2^{t(t+1) + 8(t^3 - t^2 + t - 1 + \frac{1}{t+1} )}(t+1) \\
 &< & 2^{8t^3 - t} \\
 & < & k - t + 1 =  h(x).
\end{eqnarray*}
Therefore, some color $c \in L(x)$ can be used to color $x$ without violating (\ref{star}).

Now, with every vertex $x_j^i \in X$ colored, and with (\ref{star}) satisfied, no index $\ell$ belongs to the intersection of more than $t$ sets $S_c$, where $c$ is a color used at a vertex $x_j^i$, and hence after coloring the vertices in $X$, at most $t$ colors are unavailable at each clique $K^{\ell}$. Therefore, the vertices of each clique $K^{\ell}$ can be ordered so that the first vertex has at least one available color, the second vertex has at least two available colors, and so forth, until the last vertex has $t+1$ available colors. Therefore, each remaining clique $K^{\ell}$ can be $DP$-colored with its available colors, and the lemma is proven.
\end{proof}

Now, we construct our graph $G_t$. First, we make $k^{k-2t}$ copies of the graph $H_t$, and we index them by the $(k-2t)$-tuples in $[k]^{k-2t}$. We also add $k-2t$ vertices $y_1, \dots, y_{k-2t}$ that are adjacent to all vertices of $U$ in each copy of $H_t$. We write $\tilde{U}$ for this set of neighbors of $y_1, \dots, y_{k-2t}$, that is, the set of vertices belonging to a set $U$ in some copy of $H_t$. The following two theorems show that $\chi_P(G_t) - \chi_{DP}(G_t) \geq t$.

\begin{theorem}
$\chi_{DP}(G_t) \leq k-t+1$.
\end{theorem}
\begin{proof}
We may give $G_t$ a DP-coloring with lists of size $(k-t+1)$ as follows. First, we arbitrarily color the vertices $y_1, \dots, y_{k-2t}$. Next, we observe that 
the vertices in $\tilde U$ have lost at most $k-2t$ available colors, so
for each vertex $v$ in a copy of $H_t$, $v$ has at least $h(v)$ available colors remaining. Therefore, by Lemma \ref{lem:yes}, we may finish our DP-coloring of $G_t$ by giving each remaining copy of $H_t$ a DP-$h$-coloring. 
\end{proof}

\begin{theorem}
\label{thm:no}
$\chi_P(G_t) > k$. 
\end{theorem}
\begin{proof}
Suppose that the online list-coloring game is played on $G_t$ with $k$ tokens at each vertex. We will show that Lister has a winning strategy.
For each pair $(i,j)$ satisfying $1 \leq i \leq k$
and
 $1 \leq j \leq k-2t$, Lister executes the following command. When Lister executes the command for a given pair $(i,j)$, we say that this takes place on \emph{turn $(i,j)$}.
\begin{quote}
 Reveal a color $c_{i,j}$ above $y_j$ and above each vertex of $\tilde{U}$ that belongs to a copy of $H_t$ indexed by a $(k-2t)$-tuple 
 with the value $i$ in the $j$th coordinate.
\end{quote}
For each value $j \in [k-2t]$, we write $L(y_j) = \{c_{1,j}, \dots, c_{k,j}\}$ for the set of colors revealed above $y_j$.

For each $j \in [k-2t]$, 
we may assume that for some value $i_{j} \in [k]$, Painter colors $y_{j}$ 
during turn $(i_{j}, j)$
and hence does not color any vertex of $\tilde U$ during turn $(i_j, j)$. 
Indeed, if this is not the case, then $y_j$ is never colored, and Painter will not have another chance to color $y_{j}$.
Therefore, 
for each value $j \in [k-2t]$, we may assume that 
no vertex in a copy of $H_t$ indexed by a $(k-2t)$-tuple with an $i_j$ entry in the $j$th coordinate is colored
using a color in $L(y_j)$.

Now, let $H$ be the copy of $H_t$ indexed by the $(k-2t)$-tuple $(i_1, \dots, i_{k-2t})$. By our observation above, no vertex of $H$ has been colored by a color in $L(y_1) \cup \dots \cup L(y_{k-2t})$, and hence no vertex of $H$ has been colored. 
Furthermore, 
since $k-2t$ tokens have been removed from each vertex in $\tilde{U} \cap V(H)$, it follows that 
for each vertex $v \in V(H)$, only $h(v)+t-1$ tokens remain at $v$. Therefore, Lister can follow the strategy in Lemma \ref{lem:no} on $H$ in order to win the game, and thus $\chi_P(G_t) > k$.
\end{proof}

\section{Conclusion}
While we have shown for each $t \geq 1$ the existence of a graph $G_t$ for which $\chi_P(G_t) - \chi_{DP}(G_t) \geq t$, it is still open whether there exists a sequence $\{G_t\}_{t \geq 1}$ of graphs for which 
\[\lim_{t \rightarrow \infty} \frac{\chi_{DPP}(G_t)}{\chi_{DP}(G_t)} > 1 \textrm{ \indent or \indent } \lim_{t \rightarrow \infty} \frac{\chi_{P}(G_t)}{\chi_{\ell}(G_t)} > 1.\] 
On the other hand, it is unknown whether $\chi_P(G)$ can be bounded above by a linear or even polynomial function of $\chi_{\ell}(G)$, and it is unknown whether $\chi_{DPP}(G)$ can be bounded above by a linear function of $\chi_{DP}(G)$. Duraj, Gutowski, and Kozik \cite{Duraj} have pointed out that currently, the best known bound for $\chi_P(G)$ in terms of $\chi_{\ell}(G)$ comes from
the
relationship between a graph's choosability and minimum degree. 
Namely, a
result of Saxton and Thomason \cite{Saxton} states that a graph $G$ of minimum degree $\delta$ satisfies $\chi_{\ell}(G) \geq (1 + o(1)) \log_2 \delta$. Writing $d$ for the degeneracy of a graph $G$, we observe that $G$ must have a subgraph of minimum degree $d$, and hence we may use this result to observe that
\[\chi_P(G) \leq d+1 \leq   2^{(1+o(1))\chi_{\ell} (G)}.\]
For $\chi_{DPP}(G)$,
we may use a result of 
of Bernshteyn showing that a graph $G$ of minimum degree $\delta$ satisfies $\chi_{DP}(G) \geq \frac{\delta / 2}{\log (\delta / 2)}$ in order to bound
$\chi_{DPP}(G)$ in terms of $\chi_{DP}(G)$ in a similar way.
Using the same observation as above, if $d$ is the degeneracy of $G$, then 
\[\chi_{DPP}(G) \leq d + 1\leq (2 + o(1)) \chi_{DP}(G)  \log \chi_{DP}(G).\]
 It is likely that a deeper understanding of these coloring parameters is necessary to determine tight bounds between them.

\section{Acknowledgment}
I am grateful to Alexandr Kostochka for offering helpful feedback on an earlier version of this paper.

\raggedright
\bibliographystyle{plain}
\bibliography{CCbib}

\end{document}